\documentclass[12pt]{amsart}
\usepackage{amsfonts}
\usepackage{geometry}
\usepackage{graphicx}
\usepackage{amssymb}
\usepackage{ulem}
\usepackage[colorlinks=true]{hyperref}
\usepackage{enumerate}

\numberwithin{equation}{section}
\hypersetup{urlcolor=blue, citecolor=blue}

\newtheorem{theorem}{Theorem}[section]

\newtheorem{lemma}[theorem]{Lemma}

\theoremstyle{definition}
\newtheorem{definition}[theorem]{Definition}
\newtheorem{remark}{Remark}

\def\dfrac{\displaystyle\frac}
\def\dint{\displaystyle\int}

\newcommand{\Ui}{\partial U_t^{int}}
\newcommand{\Ue}{\partial U_t^{ext}}


\begin{document}
\title[Sharp estimates for solutions to elliptic problems with mixed BCs]{Sharp estimates for solutions to elliptic problems with mixed boundary conditions}
\author[ A. Alvino, F. Chiacchio, C. Nitsch, C. Trombetti]{A. Alvino, F. Chiacchio, C. Nitsch, C. Trombetti}
\date{}
\address{\vskip1cm\noindent \hfill\break\vskip-.2cm \noindent \hfill\break
\vskip-.2cm \noindent Dipartimento di Matematica e Applicazioni ``R.
Caccioppoli'', Universit\`{a} degli Studi di Napoli ``Federico II'',
Complesso Universitario Monte S. Angelo, via Cintia - 80126 Napoli, Italy.
\hfill\break\vskip-.2cm \noindent e-mail: \texttt{angelo.alvino@unina.it, fchiacch@unina.it, c.nitsch@unina.it, cristina@unina.it}}
\maketitle

\begin{abstract} We show, using symmetrization techniques, that it is possible to prove a comparison principle (we are mainly focused on $L^1$ comparison) between solutions to an elliptic partial differential equation on a smooth bounded set $\Omega$ with a rather general boundary condition, and solutions to a suitable related problem defined on a ball having the same volume as $\Omega$. This includes for instance mixed problems where Dirichlet boundary conditions are prescribed on part of the boundary, while Robin boundary conditions are prescribed on its complement.
\end{abstract}

\section{Introduction}
In a recent paper \cite{ANT}, a comparison principle for solution to elliptic partial differential equation with Robin boundary conditions was exploited for the first time using symmetrization techniques. This was for long time believed to be impossible in view of the lack of a Polya Sz\"ego principle for Sobolev functions defined on a bounded domain and not assuming constant value on its boundary. Nontheless the comparison is quite sensitive to the dimension and contrary to the classical Talenti's principle \cite{Ta} depends upon the source term. This seems to be a distinctive feature of Robin problems and makes the estimates rather difficult to obtain. 

Here we push our analysis even further and we consider a generalized Robin boundary condition. Aiming at filling the gap between Robin and Dirichlet, we consider the second one as a special case of the first one when the boundary parameter blows-up.

The outcome is a comparison result for special classes of problems where Robin and Dirichlet conditions can be mixed up.  
Let $\Omega $ be an open, bounded set of $\mathbb{R}^{N}$ with Lipschitz
boundary. Let $\beta (x)$ be a measurable function defined on $
\partial \Omega $ such that $0<m<\beta(x) \le M$ and $f\in L^2(\Omega)$ a non negative function. We consider the following problem 
\begin{equation}
\left\{ 
\begin{array}{ll}
-\Delta u=f & \mbox{in $\Omega$} \\ 
&  \\ 
\displaystyle\frac{\partial u}{\partial \nu }+\beta (x)\,u=0 & 
\mbox{on
$\partial\Omega$,}
\end{array}
\right.   \label{problem}
\end{equation}
where $\nu $, denotes the outer unit normal to $\partial \Omega $.

A function $u \in H^1(\Omega)$ is a weak solution to \eqref{problem} if

\begin{equation}
\int_{\Omega }\nabla u\nabla \phi \,dx+\int_{\partial \Omega }\beta
(x)\,u\phi \,d\mathcal{H}^{N-1}(x)=\int_{\Omega }f\phi \,dx\quad \forall \phi
\in H^{1}(\Omega ).  \label{weaksol}
\end{equation}
We will establish a comparison principle with the solution to the following
problem 
\begin{equation}
\left\{ 
\begin{array}{ll}
-\Delta v=f^\sharp & \mbox{in $\Omega^\sharp$} \\ 
&  \\ 
\displaystyle\frac{\partial v}{\partial \nu }+\overline{\beta }\,v=0 & 
\mbox{on $\partial\Omega^\sharp$.}
\end{array}
\right.   \label{problem_sharp}
\end{equation}
where $\Omega ^{\sharp }$ denotes the ball, centered at the origin, with the
same Lebesgue measure as $\Omega $, $f^\sharp$ is the Schwarz rearrangement of $f$, and $\overline{\beta }\,>0$ is a constant defined by
the following relation
\begin{equation}\label{main_cond}
\frac{Per( \Omega^\sharp)}{\overline \beta} =\left(\int_{\partial \Omega ^{\sharp }}\frac{1}{\overline{\beta }}d\mathcal{H}
^{N-1}(x)\right){=}\int_{\partial \Omega }\frac{1}{\beta (x)}d\mathcal{H}^{N-1}(x).
\end{equation}
Our main theorem is

\begin{theorem}\label{th_main_1} 
Let $u$ and $v$ be the solutions to Problem \eqref{problem}
and to Problem \eqref{problem_sharp}, respectively. Then,
when $N = 2$, we have
$$\|u\|_{L^{1}(\Omega)}\le \|v\|_{L^{1}(\Omega^\sharp)}.$$
While for $N\ge 3$
$$\|u\|_{L^{1}(\Omega)}\le \|v\|_{L^{1}(\Omega^\sharp)},$$
provided
\begin{equation}\label{f_cond}
\int_E f(x)dx \le \frac{|E|^{1-\frac{2}N}}{|\Omega|^{1-\frac{2}N}}\int_\Omega f(x)dx
\end{equation}
for all measurable $E\subseteq\Omega$.
Moreover for $N=2$ and $f\equiv 1$ we have
\begin{equation*}
u^{\sharp }(x)\leq v(x)\quad x\in \>\Omega ^{\sharp }.
\end{equation*}
\end{theorem}

\begin{remark}
We observe that without \eqref{f_cond} Theorem \ref{th_main_1} is false (see Remark \ref{rem_count}). Condition \eqref{f_cond} is fulfilled for instance by $f$ constant. 
If we consider the decreasing rearrangement of $f$, namely $f^*$ (see the next section for the definition), 
it reads as follows
$$s^{1-\frac{2}{N}}\int_0^s f^*(r)dr \le |\Omega|^{1-\frac{2}{N}}\int_0^{|\Omega|} f^*(r)dr.$$
Whether or not \eqref{f_cond} is optimal is still an open problem.
\end{remark}

Robin boundary conditions with a variable boundary parameter appeared for instance in \cite{BBN} in the context of optimal insulation and in \cite{DGK} where the  optimization of the $p$-Laplacian eigenvalue with respect to the boundary parameters is shown to be a well posed problem.

The upper bound $M$ on the function $\beta$ is a technical assumption which simplifies the computation, but can be easily relaxed. In fact, once the estimates in Theorem \ref{th_main_1} have been established, by continuity argument they are comfortably extended to cases where $\beta=+\infty$ on part of the boundary of $\Omega$ (and in this case we formally consider $1/\beta=0$ on that portion). This account for comparison principles for solutions to elliptic problems with mixed boundary conditions, a result that to our knowledge is completely new.

To give an example, we can consider an open bounded set $\Omega$ whose boundary is, up to a negligeble set, the union of two smooth manifolds $\Gamma_0$ and $\Gamma_1$ and consider the case $\beta=\hat\beta = const.>0$ on $\Gamma_0$ and $\beta=+\infty$ on $\Gamma_1$. In such a case we have

\begin{equation}
\left\{ 
\begin{array}{ll}
-\Delta u=f & \mbox{in $\Omega$,} \\ 
&  \\ 
\displaystyle\frac{\partial u}{\partial \nu }+\hat\beta\,u=0 & 
\mbox{on
$\Gamma_0$,}\\ &\\ 
\displaystyle u=0 & 
\mbox{on
$\Gamma_1$.}
\end{array}
\right.
\end{equation}

The resulting solution can be compared, in view of Theorem \ref{th_main_1}, 
with the solution to the following problem
\begin{equation}
\left\{ 
\begin{array}{ll}
-\Delta u=f^\sharp & \mbox{in $\Omega^\sharp$,} \\ 
&  \\ 
\displaystyle\frac{\partial u}{\partial \nu }+\overline\beta\,u=0 & 
\mbox{on
$\partial \Omega^\sharp$,}
\end{array}
\right.
\end{equation}
where the constant $\displaystyle\overline\beta=\frac{Per(\Omega^\sharp)}{\mathcal{H}
^{N-1}(\Gamma_0)}\hat\beta$.

%

\section{Notation and Preliminaries}\label{sec_not}

The solution $u \in H^1(\Omega)$ to \eqref{problem} is the unique minimizer
of

\begin{equation}
\min_{w\in H^{1}(\Omega )}{\frac{1}{2}\int_{\Omega }|\nabla w|^{2}\,dx+\frac{
1}{2}\int_{\partial \Omega }\beta (x)w^{2}\,}d\mathcal{H}^{N-1}(x)-{
\int_{\Omega }w\,dx}.  \label{minimizer}
\end{equation}


For $\displaystyle t\ge 0$ we denote by

\begin{equation*}
U_t=\{x \in \Omega: u(x)>t\}, \quad \partial U_t^{int} = \partial U_t \cap
\Omega, \quad \partial U_t^{ext}= \partial U_t \cap \partial\Omega,
\end{equation*}

and by 
\begin{equation*}
\mu (t)=|U_{t}|,\quad P_{u}(t)=\mathrm{Per}(U_{t}).
\end{equation*}
the Lebesgue measure of $U_{t}$ and its perimeter in $\mathbb{R}^{N}$,
respectively. 
Moreover, $\Omega ^{\sharp }$ denotes the ball, centered at the origin, with
the same measure as $\Omega $ and $v$ denotes the unique, radial and
decreasing along the radius, solution to Problem\eqref{problem_sharp}.


Then, using the same notation as above, for $t\ge 0$ we set

\begin{equation*}
V_t=\{x \in \Omega^\sharp: v(x)>t\}, \quad \phi(t) = |V_t|, \quad \quad
P_v(t) = \mathrm{Per}(V_t).
\end{equation*}
Since $v$ is radial, positive and decreasing along the radius then, for $
0\le t\le\min_{\Omega^\sharp} v$, $V_t$ coincides with $\Omega^\sharp$,
while, for $\min_{\Omega^\sharp} v <t<\max_{\Omega^\sharp} v$, $V_t$ is a
ball concentric to $\Omega^\sharp$ and strictly contained in it.

In what follows we denote by $\omega_N$ the measure of the unit ball in $
\mathbb{R}^N$.

\begin{definition}
\label{rarrangement} Let $h:x\in \Omega \rightarrow \lbrack 0,+\infty
\lbrack $ be a measurable function, then the decreasing rearrangement $
h^{\ast }$ of $h$ is defined as follows: 
\begin{equation*}
h^{\ast }(s)=\inf \{t\geq 0:|\{x\in \Omega :|h(x)|>t\}|<s\}\quad s\in
\lbrack 0,\Omega ].
\end{equation*}%
while the Schwarz rearrangement of $h$ is defined as follows 
\begin{equation*}
h^{\sharp }(x)=h^{\ast }(\omega _{N}|x|^{N})\quad x\in \Omega ^{\sharp }.
\end{equation*}
\end{definition}

It is easily checked that $h$, $h^*$ and $h^\sharp$ a are equi-distributed,
i.e. 
\begin{equation*}
|\{x\in\Omega: |h(x)|>t\}| = |\{s\in (0,|\Omega|: h^*(s)>t\}| =
|\{x\in\Omega^\sharp: h^\sharp(x)>t\}|\quad t\ge 0
\end{equation*}
and then 
if $h\in L^p(\Omega)$, $1 \le p \le \infty$, then $h^* \in L^p(0,|\Omega|)$, 
$h^\sharp \in L^p(\Omega^\sharp)$, and 
\begin{equation*}
||h||_{L^p(\Omega)}=||h^*||_{L^p(0,|\Omega|)}=||h^\sharp||_{L^p(\Omega^
\sharp)} .
\end{equation*}

\section{Proof of Theorem \protect\ref{th_main_1}}\label{sec_main}

The main ingredient for a comparison result is the following lemma.

\begin{lemma}
Let $u$ and $v$ be the solution to \eqref{problem} and \eqref{problem_sharp}
, respectively. For a.e. $t>0$ we have 
\begin{equation}
\gamma _{N}\phi (t)^{2-\frac{2}{N}}=\int_0^{\phi(t)}f^*(s)\,ds\left(-\phi ^{\prime }(t)+\int_{\partial
V_{t}\cap \partial \Omega ^{\sharp }}\frac{1}{\overline{\beta }}\frac{1}{v(x)
}\>d\mathcal{H}^{N-1}(x)\right),  \label{eq_fundamental}
\end{equation}
while for almost all $t>0$ it holds 
\begin{equation}
\gamma _{N}\mu (t)^{2-\frac{2}{N}}\leq \int_0^{\phi(t)}f^*(s)\,ds\left(-\mu ^{\prime }(t)+\int_{\partial
U_{t}^{ext}}\frac{1}{\beta (x)}\frac{1}{u(x)}\>d\mathcal{H}^{N-1}(x)\right)
\label{ineq_fundamental}
\end{equation}
where $\gamma _{N}=N^{2}\omega _{N}^{-\frac{2}{N}}.$
\end{lemma}

\noindent \textbf{Proof} Let $t>0$ and $h>0$, and let us choose the
following test function in \eqref{weaksol} 
\begin{equation}
\varphi _{h}(x)=\left\{ 
\begin{array}{ll}
0 & \mbox{if $0<u<t$} \\ 
&  \\ 
h & \mbox{if $u> t+h$} \\ 
&  \\ 
u-t & \mbox{if $t<u<t+h$}.
\end{array}
\right.
\end{equation}

Then,

\begin{equation}
\begin{array}{ll}
\displaystyle\int_{U_{t}\setminus U_{t+h}}|\nabla u|^{2}\,dx&+ h\displaystyle
\int_{\partial U_{t+h}^{ext}}\beta u\,d\mathcal{H}^{N-1}(x) \\\displaystyle
&+\displaystyle\int_{\partial U_{t}^{ext}\setminus \partial U_{t+h}^{ext}}\beta u(u-t)\,d
\mathcal{H}^{N-1}(x)
\\&= \displaystyle\int_{U_{t}\setminus U_{t+h}}(u-t)\,dx+h\displaystyle
\int_{U_{t+h}}f(x)\,dx
\end{array}
\end{equation}
dividing by $h$ and letting $h$ go to zero, using coarea formula we have
that for a.e. $t>0$

\begin{eqnarray}
\int_{\partial U_{t}}g(x)\,d\mathcal{H}^{N-1}(x)&=&\int_{\partial
U_{t}^{int}}|\nabla u|\,d\mathcal{H}^{N-1}(x)\\&&+\int_{\partial
U_{t}^{ext}}\beta ud\mathcal{H}^{N-1}(x)=\int_{U_{t}}f(x)\,dx\notag
\end{eqnarray}

where 
\begin{equation*}
g(x)=\left\{ 
\begin{array}{ll}
|\nabla u| & \mbox{if $x \in \Ui$} \\ 
&  \\ 
\beta u & \mbox{if $x \in \Ue$}
\end{array}
\right.
\end{equation*}
for a.e. $t>0$ we have 
\begin{equation}
\begin{array}{ll}
&P_{u}^{2}(t)  \leq \left( \dint_{\partial U_{t}}g(x)d\mathcal{H}
^{N-1}(x)\right) \left( \dint_{\partial U_{t}}g(x)^{-1}d\mathcal{H}
^{N-1}(x)\right) = \\ 
&  \\ 
& \left( \dint_{\partial U_{t}}g(x)d\mathcal{H}^{N-1}(x)\right) \left(
\dint_{\partial U_{t}^{int}}|\nabla u|^{-1}d\mathcal{H}^{N-1}(x)+\dint_{
\partial U_{t}^{ext}}(\beta u)^{-1}d\mathcal{H}^{N-1}(x)\right) \leq \\ 
&  \\ 
& \dint_0^{\mu(t)}f^*(s)\,ds \left( -\mu ^{\prime }(t)+\dint_{\partial U_{t}^{ext}}(\beta
u)^{-1}\>d\mathcal{H}^{N-1}(x)\right) \quad t\in \lbrack 0,\max_{\Omega }u).
\end{array}
\label{estimate1}
\end{equation}
Then the isoperimetric inequality ($P_{u}(t)\geq N\omega _{N}^{\frac{1}{N}
}\mu (t)^{1-\frac{1}{N}}$) gives
\begin{equation*}
\gamma _{N}\mu (t)^{2-\frac{2}{N}}\leq \int_0^{\mu(t)}f^*(s)\,ds\left(-\mu ^{\prime }(t)+\int_{\partial
U_{t}^{ext}}\frac{1}{\beta (x)u(x)}\>d\mathcal{H}^{N-1}(x)\right)\quad t\in \lbrack
0,\max_{\Omega }u).
\end{equation*}
If $v$ solves Problem \eqref{problem}, all the previous inequalities hold as
equalities, hence \eqref{eq_fundamental} follows.$\hfill\square $

\begin{remark}
We observe that solutions $u$ and $v$ to Problem \eqref{problem} and Problem 
\eqref{problem_sharp}, always achieve their minima on the boundary of $
\Omega $ and $\Omega ^{\sharp }$ respectively. From now on we denote by 
\begin{equation*}
v_{m}=\min_{\Omega ^{\sharp }}v,\>u_{m}=\min_{\Omega }u.
\end{equation*}
The following inequality holds true 
\begin{equation}
u_{m}\leq v_{m}.  \label{ineq_bordo}
\end{equation}
In fact, using Schwarz inequality, the equations with the boundary conditions in \eqref{problem} and the isoperimetric inequality
and \eqref{problem_sharp},
\begin{eqnarray*}
\sqrt{u_{m}}\mathrm{Per}(\Omega ) &\leq &\int_{\partial \Omega }\>\sqrt{
\beta (x)}\text{ }\sqrt{u(x)}\frac{1}{\sqrt{\beta (x)}}d\mathcal{H}^{N-1}(x) \\
&\leq &\left( \int_{\partial \Omega }\>\frac{1}{\beta (x)}\text{ }d\mathcal{H
}^{N-1}(x)\right) ^{\frac{1}{2}}\left( \int_{\partial \Omega }u(x)\>\beta (x)
\text{ }d\mathcal{H}^{N-1}(x)\right) ^{\frac{1}{2}} \\
&=&\left( \int_{\partial \Omega ^{\sharp }}\frac{1}{\overline{\beta }}{d{
\mathcal{H}}^{N-1}}(x)\right) ^{\frac{1}{2}}\left( \int_{\partial \Omega
}u(x)\>\beta (x)\text{ }d\mathcal{H}^{N-1}(x)\right) ^{\frac{1}{2}} \\
&=&\left( \int_{\partial \Omega ^{\sharp }}\frac{1}{\overline{\beta }}{d{
\mathcal{H}}^{N-1}}(x)\right) ^{\frac{1}{2}}\left( \int_{\partial \Omega
^{\sharp }}v(x)\>\overline{\beta }\text{ }d\mathcal{H}^{N-1}(x)\right) ^{\frac{
1}{2}} \\
&=&\sqrt{v_{m}}\mathrm{Per}(\Omega ^{\sharp })\leq \sqrt{v_{m}}\mathrm{Per}
(\Omega ).
\end{eqnarray*}
\end{remark}

\bigskip

An  consequence of \eqref{ineq_bordo}, is that 
\begin{equation}
\mu (t)\leq \phi (t)=|\Omega |\qquad \mbox{for all $0\le t \le v_m$.}
\label{ineq_iniziale}
\end{equation}%
With strict inequality for some $0\leq t\leq v_{m}$ unless $\Omega $ is a
ball.

A fundamental lemma which allows us to estimate the boundary integral on the
right hand side on \eqref{estimate1} is the following.

\begin{lemma}
\label{lem_boundary} For all $t\geq v_{m}$ we have 
\begin{equation}
\int_{0}^{t}\left( \int_{\partial V_{\tau }\cap \partial \Omega ^{\sharp }}
\frac{1}{\overline{\beta }v(x)}\>d\mathcal{H}^{N-1}(x)\right) \,d\tau
=\int_{\partial \Omega ^{\sharp }}\frac{1}{\overline{\beta }}{d{\mathcal{H}}
^{N-1}(x)},  \label{eq_boundary}
\end{equation}
while 
\begin{equation}
\int_{0}^{t}\left( \int_{\partial U_{\tau }^{ext}}\frac{1}{\beta (x)u(x)}\>d
\mathcal{H}^{N-1}(x)\right) \,d\tau \leq \int_{\partial \Omega }\frac{1}{
\beta (x)}\>d\mathcal{H}^{N-1}(x).  \label{ineq_boundary}
\end{equation}
\end{lemma}

\begin{proof}
By Fubini's theorem and using \eqref{problem} we have
\begin{eqnarray*}
\int_{0}^{\infty }\left( \int_{\partial U_{\tau }^{ext}}\frac{1}{\beta
(x)u(x)}\>d\mathcal{H}^{N-1}(x)\right) \,d\tau  &=&\int_{\partial \Omega
}\left( \int_{0}^{u(x)}\frac{1}{\beta (x)u(x)}\>d\tau \,\right) d\mathcal{H}
^{N-1}(x) \\
&=&\int_{\partial \Omega }\frac{1}{\beta (x)}\>d\mathcal{H}^{N-1}(x)
\end{eqnarray*}
Analogously, 
\begin{equation*}
\int_{0}^{\infty }\int_{\partial V_{\tau }\cap \partial \Omega ^{\sharp }}
\frac{1}{\overline{\beta }v(x)}\>d\mathcal{H}^{N-1}(x)\,d\tau
=\int_{\partial \Omega ^{\sharp }}\frac{1}{\overline{\beta }}d\mathcal{H}
^{N-1}(x).
\end{equation*}
Therefore, one trivial inequality for $t\geq 0$ is 
\begin{equation*}
\int_{0}^{t}\int_{\partial U_{\tau }^{ext}}\frac{1}{\beta(x)u(x)}\>d\mathcal{H}
^{N-1}(x)\,d\tau \leq \int_{0}^{\infty }\int_{\partial U_{\tau }^{ext}}\frac{
1}{\beta(x)u(x)}\>d\mathcal{H}^{N-1}(x)\,d\tau ,
\end{equation*}
while we observe that for $t\geq v_{m}=\min_{\partial \Omega ^{\sharp }}v$
then $\partial V_{t}\cap \partial \Omega ^{\sharp }=\emptyset $ 
\begin{equation*}
\int_{0}^{t}\int_{\partial V_{\tau }\cap \partial \Omega ^{\sharp }}\frac{1}{\overline{\beta}
v(x)}\>d\mathcal{H}^{N-1}(x)\,d\tau =\int_{0}^{\infty }\int_{\partial
V_{\tau }\cap \partial \Omega ^{\sharp }}\frac{1}{ \overline{\beta}v(x)}\>d\mathcal{H}
^{N-1}(x)\,d\tau .
\end{equation*}
\end{proof}

\begin{remark}
By the choice of $\overline{\beta }$ the above Lemma immediately implies, for $t \ge v_m$, 
\begin{equation*}
\int_{0}^{t}\left( \int_{\partial U_{\tau }^{ext}}\frac{1}{\beta (x)u(x)}\>d
\mathcal{H}^{N-1}(x)\right) \,d\tau \leq \int_{0}^{t}\left( \int_{\partial
V_{\tau }\cap \partial \Omega ^{\sharp }}\frac{1}{\overline{\beta }v(x)}\>d
\mathcal{H}^{N-1}(x)\right) \,d\tau .
\end{equation*}
\end{remark}

\begin{proof}[Proof of  Theorem \eqref{th_main_1}]
Let us firstly consider the case $N=2$ and $f\equiv 1$. Integrating from $0$ to $\tau ,$ with $\tau
\geq v_{m},$ \ref{ineq_fundamental} and \ref{eq_fundamental} respectively we
have
\begin{equation*}
4\pi \tau +\int_0^{\tau} d \mu(t) \leq \int_{0}^{\tau}\int_{\partial U_{t}^{ext}}\frac{1}{\beta (x)}\frac{1}{u(x)}\>d\mathcal{H}
^{1}(x)dt
\end{equation*}
and then
\begin{equation*}
4\pi \tau +\mu (\tau )-\left\vert \Omega \right\vert \leq \int_{0}^{\tau
}\int_{\partial U_{t}^{ext}}\frac{1}{\beta (x)}\frac{1}{u(x)}\>d\mathcal{H}
^{1}(x)dt
\end{equation*}
while for $\phi$ we have 
\begin{equation*}
4\pi \tau +\phi (\tau )-\left\vert \Omega \right\vert =\int_{0}^{\tau
}\left( \int_{\partial V_{t}\cap \partial \Omega ^{\sharp }}\frac{1}{
\overline{\beta }v(x)}\>d\mathcal{H}^{1}(x)\right) \,dt.
\end{equation*}
Using Lemma \ref{lem_boundary} we conclude 
\begin{equation}
\mu (\tau )\leq \phi (\tau )\quad \tau \geq v_{m}.  \label{estimate7}
\end{equation}
Since \eqref{ineq_bordo} is in force, inequality \eqref{estimate7} follows
for $t\geq 0$ and the claim is proved.
Now we consider the general case $N\geq 2$ and $f$ satisfying \eqref{f_cond}. Integrating equation \eqref{ineq_fundamental}
from $0$ to some $\tau \geq v_{m}$, upon dividing by $\mu(t)^{1-\frac2N}$, we obtain 
\begin{eqnarray*}
\gamma _{N}\int_{0}^{\tau }\mu (t)dt &\leq &\int_{0}^{\tau }\left(\mu (t)^{\frac{2
}{N}-1} \left(\int_0^{\mu(t)}f^*(s)\,ds\right)\left( -\mu ^{\prime }(t)\right) \right)dt\\
&&+\int_{0}^{\tau }\mu (t)^{\frac{2}{N}-1
}\int_0^{\mu(t)}f^*(s)\,ds\left(\int_{\partial U_{t}^{ext}}\frac{1}{\beta (x)}\frac{1}{u(x)}\>d\mathcal{H}
^{N-1}(x)\right)dt  \label{ineq_fund1} \\
&\leq &\int_{0}^{\tau }\left(\mu (t)^{\frac{2}{N}-1}\int_0^{\mu(t)}f^*(s)\,ds\right)\left( -d\mu(t)\right)\\
&&+\left\vert \Omega \right\vert ^{\frac{2}{N}-1}\int_0^{|\Omega|}f^*(s)\,ds\int_{0}^{\tau
}\left(\int_{\partial U_{t}^{ext}}\frac{1}{\beta (x)}\frac{1}{u(x)}\>d\mathcal{H}
^{N-1}(x)\right)dt  \notag
\end{eqnarray*}
While for $\phi $ from identity (\ref{eq_fundamental}) we get
\begin{eqnarray*}
\gamma _{N}\int_{0}^{\tau }\phi (t)dt&=&\int_{0}^{\tau }\left(\phi (t)^{\frac{2}{N}-1
}\int_0^{\phi(t)}f^*(s)\,ds\right)\left( -d\phi (t)\right) \\
&&+\left\vert \Omega \right\vert ^{\frac{2}{N}-1}\int_0^{|\Omega|}f^*(s)\,ds \int_{0}^{\tau }\left(\int_{\partial V_{t}\cap \partial \Omega ^{\sharp }}
\frac{1}{\overline{\beta }}\frac{1}{v(x)}\>d\mathcal{H}^{N-1}(x) \right)dt
\end{eqnarray*}
Then we use Lemma \ref{lem_boundary} to deduce 
\begin{eqnarray*}
\gamma _{N}\int_{0}^{\tau }\mu (t)\text{ }dt-\gamma _{N}\int_{0}^{\tau }\phi
(t)dt &\leq &-\int_{0}^{\tau }\mu (t)^{\frac{2}{N}-1} \left(\int_0^{\mu(t)}f^*(s)\,ds\right) \,d\mu(t)\\
&&+\int_{0}^{\tau }\phi (t)^{\frac{2}{N}-1} \left(\int_0^{\phi(t)}f^*(s)\,ds \right) \, d\phi(t) \\
&=&-F(\mu(t))+F(\phi(t)),
\end{eqnarray*}
where $F$ is the monotone increasing function defined by
$$F(s)=\int_{0}^{s }\sigma^{\frac{2}{N}-1} \int_0^{\sigma}f^*(r)\,dr \,d\sigma.$$
Setting
\begin{equation*}
U(\tau )=\int_{0}^{\tau }\mu (t)dt\text{ \ and \ }V(\tau )=\int_{0}^{\tau
}\phi (t)dt
\end{equation*}
we get
\begin{equation}
\frac{N+2}{N}\gamma _{N}\left( U(\tau )-V(\tau )\right) \leq -F\left(
U^{\prime }(\tau )\right)+F\left( V^{\prime }(\tau )\right).  \label{U-V}
\end{equation}

The last inequality easily implies that 
\begin{equation*}
U(\tau )\leq V(\tau ).
\end{equation*}
Indeed by contradiction suppose that $\exists \tau _{0}>0:$
\begin{equation*}
U(\tau _{0})-V(\tau _{0})>0.
\end{equation*}
There exists $\widehat{\tau}< \tau _{0}$ such that
\begin{equation}
U(\tau )-V(\tau )>0\text{ in }\left( \widehat{\tau },\tau _{0}\right) 
\label{absurd_UV}
\end{equation}
and 
\begin{equation}
U(\widehat{\tau })-V(\widehat{\tau })=0.  \label{tau_bar}
\end{equation}
Using (\ref{U-V}) and the monotonicity of $F$ we would have 
\begin{equation*}
U^{\prime }(\tau )-V^{\prime }(\tau )<0\text{ \ in \ }\left( \widehat{\tau }
,\tau _{0}\right),
\end{equation*}
which is a contradiction.
\end{proof}

\begin{remark}\label{rem_count}
The following example shows that Theorem \ref{th_main_1} can not
hod true just assuming that $f(x)\in L^{2}(\Omega )$. An additional condition,
like \eqref{f_cond}, must be imposed. 
For any $0<r<1$, let us consider first the following problem with singular datum. 

\begin{equation}
\left\{ 
\begin{array}{ccc}
-\Delta u=n(n-2)\omega _{n}\delta (x) & \text{in} & B_{r,0}\cup B_{R} \\ 
&  &  \\ 
\frac{\partial u}{\partial \nu }+
r^{n-1} u=0 & \text{on} & \partial
B_{r,0} \\ 
&  &  \\ 
u=0 & \text{on} & \partial B_{R},
\end{array}
\right.   \label{Ex}
\end{equation}
where $B_{R}$ and $B_{r,0}$ are two disjoint balls, $B_{r,0}$ centered in the origin, with $r^n+R^n=1$.
The corresponding symmetrized problem associated to (\ref{Ex}) by Theorem \ref{th_main_1} is the
following

\begin{equation*}
\left\{ 
\begin{array}{ccc}
-\Delta v=n(n-2)\omega _{n}\delta (x) & \text{in} & B_{1}(0) \\ 
&  &  \\ 
\frac{\partial u}{\partial \nu }+v=0 & \text{on} & \partial
B_{1}(0).
\end{array}
\right. 
\end{equation*}
Note that
\begin{equation*}
u=\left\{ 
\begin{array}{ccc}
\dfrac{1}{\left\vert x\right\vert ^{n-2}}+\frac{n-2}{r^{2n-2}} - \frac{1}{r^{n-1}} & \text{in} & B_{r,0} \\ 
&  &  \\ 
0 & \text{in} & B_{R}
\end{array}
\right. 
\end{equation*}
and
\begin{equation*}
v=\frac{1}{\left\vert x\right\vert ^{n-2}}+n-3\text{ in }B_{1}.
\end{equation*}

From the previous consideration we deduce that Theorem 1.1 cannot hold true
\ in this example. Indeed we have
\begin{equation*}
\int_{B_{r}(0)}udx\geq C(r)\omega _{n}r^{n}=\frac{(n-2)\omega _{n}}{r^{n-2}}
-\omega _{n}r^{2}.
\end{equation*}
Therefore as $r\rightarrow 0^{+}$
the $L^{1}$-norm of $u$ diverges while the one of $v$ does not depend on $r$.
It is clear that, by approximating the $\delta$,  one can build, from problem \eqref{Ex}, counter-examples with smooth data.
\end{remark}

\section*{Acknowledgements}
This work has been partially supported by a MIUR-PRIN 2017 grant “Qualitative and
quantitative aspects of nonlinear PDE’s,  2017JPCAPN” and by GNAMPA of INdAM.


\end{document}